\documentclass[12pt]{article}
\usepackage{amsmath}
\usepackage{amsfonts}
\usepackage{amssymb}
\usepackage{eufrak}
\usepackage{color}
\usepackage{srcltx}
\usepackage[sans]{dsfont}
\usepackage{pdfsync}

\newcommand{\cA}{{\mathfrak A}}
\newcommand{\cS}{{\mathfrak S}}

\newcommand{\cP}{{\mathfrak P}}

\newcommand{\cM}{{\mathfrak M}}

\newcommand{\cG}{{\mathfrak G}}

\newcommand{\cF}{{\mathfrak F}}
\newcommand{\cC}{{\mathfrak C}}

\newcommand{\cU}{{\mathfrak U}}
\newcommand{\cT}{{\Bbb{T}}}

\newtheorem{theorem}{Theorem}[section]

\newtheorem{proposition}{Proposition}[section]
\newtheorem{example}{Example}
\newtheorem{rem}{Remark}[section]
\newtheorem{definition}{Definition}[section]

\newtheorem{cor}{Corollary}[section]

\newcommand{\bgeqn}{\begin{eqnarray}}
\newcommand{\edeqn}{\end{eqnarray}}
\newcommand{\bgeq}{\begin{eqnarray*}}
\newcommand{\edeq}{\end{eqnarray*}}
\newcommand{\bec}{\begin{center}}
\newcommand{\enc}{\end{center}}

\def\ex{\mathop{\rm Exp}}

\newcommand{\D}{{\cal D}}

\newcommand{\V}{{\cal V}}
\newcommand{\Z}{{\cal Z}}
\newcommand{\F}{{\cal F}}

\newcommand{\G}{{\cal G}}

\newcommand{\M}{{\cal M}}

\newcommand{\OO}{{\cal O}}

\newcommand{\fb}{\rule[-2pt]{4pt}{8pt}}

\newcommand{\half}{ \mbox{\small$\frac{1}{2}$}}

\newcommand{\be}{\begin{equation}}
\newcommand{\ee}{\end{equation}}

\newcommand{\vr}{{\varrho}}

\newcommand{\lan}{\langle}
\newcommand{\ran}{\rangle}

\def\w{\omega}
\def\e{\varepsilon}
\def\O{\Omega}

\def\fo {{\preceq}}			
\def\fos {{\prec}}			
\def\so {{\preccurlyeq}}	

\newcommand{\avr}{{\sf AV@R}}
\newcommand{\AVaR}{{\sf AV@R}}
\newcommand{\VaR}{{\sf V@R}}

\def\bbr{{\Bbb{R}}} 
\def\bbe{{\Bbb{E}}} 

\newcommand{\ind}{{\mbox{\boldmath$1$}}}

\newtheorem{remark}{Remark}

\begin{document}

\title{\bf Uniqueness of Kusuoka Representations}

\author{
{\bf Alois Pichler}\\
Department of Statistics and Operations Research\\
University of Vienna\\
Vienna, Austria
 \and
{\bf Alexander Shapiro}\thanks{Research of this author  was partly supported by the NSF award CMMI 	1232623.}\\
School of Industrial and Systems Engineering\\
Georgia Institute of Technology\\
Atlanta, GA 30332-0205\\
}
\date{\today}

\maketitle

\begin{abstract}
This paper addresses law invariant  coherent  risk measures and their Kusuoka representations. By elaborating the existence of a minimal representation  we show that every Kusuoka representation can be reduced to its minimal representation. Uniqueness~-- in a sense specified in the paper -- of the risk measure's Kusuoka representation is derived from this initial result.
The Kusuoka representation is usually given   for nonatomic probability spaces. We also discuss Kusuoka representations for spaces with atoms.

Further, stochastic order relations are employed to identify the minimal Kusuoka representation. It is shown that measures in the minimal representation are extremal with respect to the order relations.
The tools are finally employed to provide the minimal representation for important practical examples.
\end{abstract}


\section{Introduction}
\label{sec:Introduction}
\setcounter{equation}{0}

This paper addresses Kusuoka representations \cite{Kusuoka} of  law invariant, coherent  risk measures. In a sense such representations are natural and useful in various applications of risk measures. Original derivations were performed in \cite{Kusuoka, Schachermayer} in $L_\infty(\O,\F,P)$ spaces. For an analysis in $L_p(\O,\F,P)$, $p\in [1,\infty)$, spaces we may refer to
Pflug and  R\"omisch \cite{pfl07} (cf.~also \cite{Filipovic2012, Grechuk2012}). It is known that Kusuoka representations are not unique (cf.~\cite{pfl07}). This raises the question of    in some sense   minimality of  such representations. It was shown in \cite{mor12} that indeed for spectral risk measures, the Kusuoka representation with a single probability measure is unique. It was also posed a question whether such minimal representation is unique in general? In this paper we give a positive answer to this question and show how such minimal representations can be derived in a constructive way. We also discuss Kusuoka representations for spaces which are not nonatomic.

\medskip
Throughout the paper we work with   spaces $\Z:=L_p(\O,\F,P)$, $p\in [1,\infty)$, rather than $L_\infty(\O,\F,P)$ spaces. That is, $Z\in \Z$ can be viewed as a random variable with finite $p$-th order moment with respect to the reference probability distribution $P$. The space $\Z$ equipped with the respective norm is a Banach space with the dual space of continuous linear functionals  $\Z^*:=L_q(\O,\F,P)$, where $q\in (1,\infty]$ and $1/p+1/q=1$.

It is said that a (real valued) functional  $\rho:\Z\to \bbr$ is a coherent risk measure if it satisfies the following axioms (Artzner et al.~\cite{ADEH:1999}):

\begin{description}
\item[{\rm (A1)}]
 {\textsl Monotonicity:} If $Z,Z'\in\Z$ and  $Z\succeq Z'$,
then $\rho(Z)\geq \rho(Z')$.
\item [{\rm (A2)}]
{\textsl Convexity:}
\[
\rho(tZ  + (1-t)Z') \le t \rho(Z) +
(1-t)\rho(Z')
\]
\
 for all $Z,Z'\in\Z$ and all $t\in [0,1]$.
\item[{\rm (A3)}]  {\textsl Translation Equivariance:} If $a\in \bbr$ and $Z\in\Z$, then $\rho(Z+a)=\rho(Z)+a$.
\item [{\rm (A4)}]  {\textsl Positive Homogeneity:} If $t\ge 0$ and $Z\in \Z$, then
$\rho(t Z)=t\rho(Z)$.
\end{description}
The notation $Z \succeq Z'$ means that $Z(\w)\geq Z'(\w)$ for a.e. $\w\in \Omega$.
For a thorough discussion of coherent risk measures we refer to F\"ollmer and   Schield \cite{fol04}.

We say that $Z_1,Z_2\in \Z$ are {\em distributionally equivalent} if
$F_{Z_1}=F_{Z_2}$, where $F_Z(z):=P(Z\le z)$ denotes the cumulative distribution function (cdf) of $Z\in \Z$.
It is said that risk measure $\rho:\Z\to\bbr$ is {\em law invariant} (with respect to the reference probability measure $P$)  if
for any distributionally equivalent  $Z_1,Z_2\in \Z$ it follows that $\rho(Z_1)=\rho(Z_2)$. {\em Unless stated otherwise we assume that $\rho$ is a (real valued)  law invariant coherent risk measure.} An important example of a law invariant coherent risk measure is
 the (upper) Average Value-at-Risk
\begin{equation}\label{2k}
\avr_\alpha(Z):=\inf_{t\in \bbr}\big\{t+(1-\alpha)^{-1}\bbe [Z-t]_+\big\}=(1-\alpha)^{-1}\int_\alpha^1 F_Z^{-1}(\tau)d\tau,
\end{equation}
where
$
F_Z^{-1}(\tau):=\sup\{t:F_Z(t)\le \tau\}
$
is the right   side quantile function. Note that $F_Z^{-1}(\cdot)$ is a monotonically nondecreasing right side continuous function.
We denote by $\cP$ the set of probability measures on $[0,1]$ having zero mass at 1. When talking about topological properties of $\cP$ we use    the {\em  weak topology} of probability  measures. Note that by Prohorov's theorem the set $\cP$ is compact, as it is closed and  tight \cite[p.59]{billin}.

We say that a set $\cM$ of probability measures on $[0,1]$   is a
{\em Kusuoka set} if the following representation holds
\begin{equation}\label{1k}
 \rho(Z)=\sup_{\mu\in \cM}\int_0^1 \avr_\alpha (Z)d\mu(\alpha),\;\;Z\in \Z.
\end{equation}
 Note that since $\rho(Z)$ is finite valued for every $Z\in L_p(\O,\F,P)$, with  $p\in [1,\infty)$, every measure  $\mu\in \cM$  in representation (\ref{1k}) has zero mass at $\alpha=1$ and hence $\cM\subset \cP$. Note also that if $\cM$    is a  Kusuoka set, then its topological closure is also   a Kusuoka set (cf.~\cite[Proposition 1]{mor12}).

\medskip
The paper is organized as follows. In the next section we discuss minimality and uniqueness of the Kusuoka representations. In Section~\ref{sec:general} we consider Kusuoka representations on general, not necessarily nonatomic, probability spaces. In Section~\ref{sec;maxim} we investigate  maximality of Kusuoka representations with respect to some order relations. In Section~\ref{sec:examp} we discuss some examples, while Section~\ref{sec:concl} is devoted to conclusions.

\section{Uniqueness of Kusuoka sets}
\label{sec:Kusuoka}
\setcounter{equation}{0}

It is known that the Kusuoka representation is not unique in general. In this section we elaborate that there is minimal Kusuoka representation in the sense outlined below. To obtain the result we shall relate the Kusuoka representation to spectral risk measures first and then outline the results.

\medskip
We can view the integral in the right hand side of (\ref{1k}) as the Lebesgue-Stieltjes integral with $\mu:\bbr\to [0,1]$ being a right side continuous monotonically nondecreasing function (distribution function) such that  $\mu(t)=0$ for $t<0$ and $\mu(t)=1$ for $t\ge 1$.
For example, take $\mu(t)=0$ for $t<0$ and $\mu(t)=1$ for $t\ge 0$. This is a distribution  function corresponding to a measure of mass one at $\alpha=0$.  Thus \begin{equation}\label{integr}
\int_0^1 \avr_\alpha (Z)d\mu(\alpha)=\avr_0(Z)=\bbe[Z],
\end{equation}
where the integral is understood as taken from $0^-$ to 1.
\paragraph{Note:} When considering integrals of the form $\int_0^\gamma h(\alpha)d\mu(\alpha)$, $\gamma\ge 0$,  with respect to a distribution function $\mu(\cdot)$ we always assume that the integral is  taken from $0^-$. \medskip

Assume that the space $(\O,\F,P)$ is nonatomic. Then, without loss of generality, we can take this space to be the interval $\O=[0,1]$ equipped with its Borel sigma algebra and the uniform probability measure $P$.
We refer to this space as the {\em standard} probability space.
Unless stated otherwise  we assume that $(\O,\F,P)$ is the standard probability space.
\begin{definition}
\label{def-group}
We denote by $\cF$ the group of measure-preserving transformations $T:\O\to\O$, i.e., $T$ is one-to-one,  onto and $P(A)=P(T(A))$ for any $A\in \F$.
\end{definition}

\begin{remark}
\label{rem-de}
{\rm
Note that    two random variables $Z_1,Z_2:\O\to\bbr$ have the same probability distribution iff there exists a measure-preserving transformation $T\in \cF$ such that\footnote{The notation  $(Z\circ T)(\w)$ stands for the composition $Z(T(\w))$).}
 $Z_2=Z_1\circ T$ (e.g., \cite{Schachermayer}).   Also
 a random variable $Z:\O\to\bbr$ is distributionally equivalent to $F_Z^{-1}$, i.e.,  there exists  $T\in \cF$ such that a.e.  $F_Z^{-1}=Z\circ T$ (e.g., \cite{mor12}).
}
\end{remark}

 \begin{definition}
 \label{def:specfun}
 We say that   $\sigma:[0,1)\to\bbr_+$ is a {\rm spectral function} if $\sigma(\cdot)$ is  right side continuous, monotonically nondecreasing and  such that $\int_0^1\sigma(t)dt=1$.
The set of spectral functions will be denoted by
  $\cS$.
\end{definition}

Consider the linear mapping  $\cT:\cP\to \cS$ defined as
\begin{equation}\label{3k}
(\cT\mu)(\tau):=\int_0^\tau (1-\alpha)^{-1}d\mu(\alpha),\;\tau\in [0,1).
\end{equation}
 This mapping is onto,  one-to-one with the inverse $\mu=\cT^{-1}\sigma$ given by
\begin{equation}\label{4k}
\mu(\alpha)=\left(\cT^{-1}\sigma\right)(\alpha)=(1-\alpha)\sigma(\alpha)+\int_0^\alpha \sigma(\tau)d\tau,\;\;\alpha\in [0,1]
\end{equation}
(cf.~\cite[Lemma 2]{mor12}). In particular, let
$\mu:=\sum_{i=1}^n c_i\delta_{\alpha_i}$, where $\delta_\alpha$ denotes the probability  measure of mass one at $\alpha$,  $c_i$ are positive numbers such that $\sum_{i=1}^n c_i=1$, and $0\le \alpha_1<\alpha_2<\cdots<\alpha_n<1$. Then \begin{equation}\label{discr1}
 \cT\mu =\sum_{i=1}^n \frac{c_i}{1-\alpha_i}\ind_{[\alpha_i,1]},
\end{equation}
where $\ind_A$ denotes the indicator function of set $A$.

The (real valued) coherent risk measure $\rho:\Z\to\bbr$ is continuous (in the norm topology of $\Z=L_p(\O,\F,P)$)  \cite{RS:2006},  and  by the Fenchel-Moreau theorem has the following dual representation
\begin{equation}\label{7k}
 \rho(Z)=\sup_{\zeta\in \cA}\,\lan \zeta,Z\ran,\;\;Z\in \Z,
\end{equation}
where $\cA\subset \Z^*$ is a convex and weakly$^*$ compact  set of density functions and the scalar product on $\Z^*\times \Z$ is defined as
$\lan \zeta,Z\ran:=\int_0^1 \zeta(\tau) Z(\tau)d\tau.$
Since $\rho$ is law invariant, the dual set $\cA$ is invariant under the group $\cF$ of measure preserving transformations, i.e., if $\zeta\in \cA$, then $\zeta\circ T\in \cA$ for any $T\in \cF$ (cf.~\cite{mor12}).

Note that if
\begin{equation}
\label{rep1}
 \rho(Z)=\sup_{\zeta\in \cC}\lan \zeta,Z\ran,\;\;Z\in \Z,
\end{equation}
holds for some set $\cC\subset \Z^*$, then $\cC\subset \cA$. For a set $\Upsilon\subset \Z^*$ we denote
\begin{equation}
\label{orbit}
\OO(\Upsilon):=\{\zeta\circ T: T\in \cF,\;\zeta\in \Upsilon\}
\end{equation}
its orbit with respect to the group $\cF$.

\begin{definition}
\label{def-genset}
We say that a set  $\Upsilon\subset \cS$ of spectral functions  is a {\em generating set} if
the representation (\ref{rep1}) holds for $\cC:= \OO(\Upsilon)$ (cf.~\cite[Definition 1]{mor12}).
That is,
\begin{equation}\label{genrepr}
 \rho(Z)=\sup_{\sigma\in \Upsilon}\int_0^1 \sigma(\tau)F_Z^{-1}(\tau)d\tau,\;\;Z\in \Z.
\end{equation}
\end{definition}

It follows  that if $\Upsilon$    is a  generating set, then
 $\OO(\Upsilon)\subset \cA$.

\begin{proposition}
\label{pr-gen}
A set $\cM\subset \cP$ is a Kusuoka set iff the set $\Upsilon:=\cT(\cM)$ is a generating set.
\end{proposition}

\begin{proof}
By using the integral representation (\ref{2k}) of $\avr$ we can write \begin{equation}\label{5k}
\int_0^1 \avr_\alpha (Z)d\mu(\alpha)=\int_0^1\int_\alpha^1 (1-\alpha)^{-1} F_Z^{-1}(\tau)d\tau d\mu(\alpha)=\int_0^1 (\cT\mu)(\tau)F_Z^{-1}(\tau)d\tau.
\end{equation}
Therefore representation (\ref{1k}) can be written as
\begin{equation}\label{6k}
 \rho(Z)=\sup_{\mu\in \cM}\int_0^1 (\cT\mu)(\tau)F_Z^{-1}(\tau)d\tau.
\end{equation}\label{k8}
Together with (\ref{genrepr}) this shows that $\cM$ is a Kusuoka set iff $\cT(\cM)$ is a generating set. \fb
\end{proof}
\\

\begin{definition}
\label{def:expos}
It is said that $\bar{\zeta}$ is a {\rm weak$^*$ exposed point} of $\cA\subset \Z^*$ if there exists $Z\in \Z$ such that  $g_Z(\zeta):=\lan \zeta, Z\ran$ attains its maximum over $\zeta\in \cA$ at the unique point $\bar{\zeta}$. In that case we say that $Z$ {\rm exposes} $\cA$ at $\bar{\zeta}$. We denote by $\ex(\cA)$ the set of exposed points of $\cA$.
\end{definition}

A result going back to  Mazur \cite{maz:33} says that if $X$ is a separable Banach space and  $K$ is a nonempty weakly$^*$ compact subset of $X^*$, then the set of points $x\in X$ which expose $K$ at some point $x^*\in K$ is a dense (in the norm topology) subset of $X$ (see, e.g.,    \cite{LAR:79} for a discussion of these type results). Since the space   $\Z=L_p(\O,\F,P)$, $p\in [1,\infty)$, is {\em separable}  we have the following theorem:

\begin{theorem}
\label{th-dens}
Let $\cA$ be the dual set in (\ref{7k}). Then the set
\begin{equation}\label{expset}
\D:=\{Z\in \Z: Z\;{\rm exposes}\;\cA\;{\rm at\; a \; point}\;\bar{\zeta}\}
\end{equation}
is   a dense (in the norm topology) subset of $\Z$ .
\end{theorem}

This allows to proceed towards a minimal representation of a risk measure as follows.

\begin{proposition}
\label{pr-rep}
Let $\ex(\cA)$ be the set of exposed points of $\cA$.
Then    the representation {\rm (\ref{rep1})}
holds with  $\cC:=\ex(\cA)$. Moreover,  if the representation {\rm (\ref{rep1})} holds for some weakly$^*$ closed set $\cC$, then $\ex(\cA)\subset\cC$.
\end{proposition}

\begin{proof}
Consider the set $\D$ defined in (\ref{expset}). By Theorem \ref{th-dens} this set is dense in $\Z$. So
for $Z\in \Z$ fixed, let $\{Z_n\}\subset \D$ be a sequence of points converging (in the norm topology) to $Z$. Let $\{\zeta_n\}\subset \ex(\cA)$ be a  sequence of the corresponding maximizers, i.e.,
$\rho(Z_n)=\lan \zeta_n,Z_n\ran$. Since $\cA$ is bounded,  we have that   $\|\zeta_n\|^*$ is uniformly bounded.
Since $\rho:\Z\to\bbr$ is real valued  it is continuous (cf.~\cite{RS:2006}), and thus  $\rho(Z_n)\to  \rho(Z)$. We also have that
\[
|\rho(Z_n)-\lan \zeta_n,Z\ran|=|\lan \zeta_n,Z_n\ran-\lan \zeta_n,Z\ran |\le \|\zeta_n\|^* \|Z_n-Z\|\to 0.
\]
It follows that
\[
\rho(Z)=\sup\{\lan \zeta_n,Z\ran: n=1,...\},
\]
and hence the representation  (\ref{rep1})
holds with  $\cC=\ex(\cA)$.

Let $\cC$ be a  weakly$^*$ closed set such that the representation  (\ref{rep1}) holds. It follows that $\cC$ is a subset of $\cA$ and is weakly$^*$ compact. Consider a point  $\zeta\in \ex(\cA)$. By the  definition of the set $\ex(\cA)$, there is $Z\in \D$ such that $\rho(Z)=\lan \zeta,Z\ran$. Since $\cC$ is weakly$^*$ compact, the maximum in (\ref{rep1}) is attained and hence
$\rho(Z)=\lan \zeta',Z\ran$ for some $\zeta'\in \cC$. By the uniqueness of the maximizer $\zeta$  it follows that $\zeta'=\zeta$. This shows that $\ex(\cA)\subset \cC$.
\fb
\end{proof}
\\

By the above proposition, the weak$^*$ closure   $\overline{\ex(\cA)}$ coincides with the intersection of all weakly$^*$ closed sets $\cC$ satisfying representation (\ref{rep1}).

\paragraph{Note:} For a set $A\subset \Z^*$ we denote by $\overline{A}$ the topological closure of $A$ in the weak$^*$ topology of the space $\Z^*$.

\begin{proposition}
 \label{pr:inv}
 The sets $\ex(\cA)$ and $\overline{\ex(\cA)}$ are invariant under the group $\cF$ of measure preserving transformations.
\end{proposition}

 \begin{proof}
  For $T\in \cF$
we have that
  \[
 \lan \zeta\circ T,Z\ran=
  \int_0^1 \zeta (T(\tau)) Z(\tau)d\tau=
  \int_0^1 \zeta (\tau) Z(T^{-1}(\tau))d\tau = \lan \zeta,Z\circ T^{-1}\ran.
  \]
Thus     $\bar{\zeta}\in \cA$ is a maximizer of $\lan \zeta,Z\ran$ over $\zeta\in \cA$, iff  $\bar{\zeta}\circ T$ is a maximizer of $\lan \zeta,T^{-1}\circ Z\ran$ over $\zeta\in \cA$.
Consequently we have that  if $\bar{\zeta}\in \ex(\cA)$, then $\bar{\zeta}\circ T\in \ex(\cA)$, i.e., $\ex(\cA)$ is invariant under $\cF$.

Now consider a point  $\bar{\zeta}\in \overline{\ex(\cA)}$. Then there is  a sequence  $\zeta_n\in \ex(\cA)$ such that
$\lan \zeta_n,Z\ran$ converges to $\lan\bar{\zeta},Z\ran$ for any $Z\in \Z$. For any $T\in \cF$ we have that $\zeta_n\circ T \in \ex(\cA)$, and hence
$
\lan \zeta_n\circ T,Z\ran = \lan \zeta_n,Z\circ T^{-1}\ran
$
converges to $\lan\bar{\zeta},Z\circ T^{-1}\ran=\lan\bar{\zeta}\circ T,Z\ran$ for any $Z\in \Z$. It follows that $\bar{\zeta}\circ T\in \overline{\ex(\cA)}$. \fb
\end{proof}
\\

By the above proposition we have that for any exposed point there is a distributional equivalent which is a monotonically nondecreasing function. We employ this observation in the following Corollary~\ref{cor-1} to reduce the generating set.

\begin{definition}
\label{def:monexp}
Denote by $\cU$ the subset of $\ex(\cA)$  formed by  right side continuous monotonically nondecreasing functions, i.e., $\cU=\cS\cap \ex(\cA)$.
\end{definition}

It follows by Proposition \ref{pr-rep} that the topological closure   $\bar{\cU}$ (of the set $\cU$)  is the  unique minimal generating set in the following sense.

\begin{cor}
\label{cor-1}
The set $\cU$ is a generating set. Moreover, if
  $\Upsilon$ is a weakly$^*$  closed generating set, then the set  $\bar{\cU}$ is a subset of $\Upsilon$, i.e., $\bar{\cU}$ coincides with the intersection of all weakly$^*$  closed generating sets.
\end{cor}

 As before, let $(\O,\F,P)$ be the standard probability space and consider   the space  $\Z:=L_p(\O,\F,P)$, $p\in [1,\infty)$,   its dual  $\Z^*=L_q(\O,\F,P)$, and the set
 \begin{equation}\label{setp}
 \cP_q:=\{\mu\in \cP:\cT\mu\in \Z^*\}.
\end{equation}

\begin{proposition}
\label{pr-con}
The set $\cP_q$ is closed and the
mapping $\cT$ is continuous on the   set  $\cP_q$ with respect to   the  weak topology of measures and the  weak$^*$ topology of $\Z^*$.
\end{proposition}

 \begin{proof}
 Let $\{\mu_k\}\subset \cP_q$ be a sequence of measures converging in the weak topology to $\mu\in \cP$, and define $\zeta_k:=\cT\mu_k$ and $\zeta :=\cT\mu$. For $Z\in \Z$
 we have that
 \[
 \int_0^1Z(t) \zeta_k(t) dt= \int_0^1\int_0^tZ(t)(1-\alpha)^{-1}d\mu_k(\alpha)dt=
  \int_0^1(1-\alpha)^{-1}h_Z(\alpha)d\mu_k(\alpha),
 \]
where $h_Z(\alpha):= \int_\alpha^1 Z(t)dt$. Note that
$|h_Z(\alpha)|\le \|Z\|_q$ for $\alpha\in[0,1]$, and
the function $\alpha\mapsto h_Z(\alpha)$ is continuous.
Let us choose a dense set $\V\subset \Z$ such that for any $Z\in \V$ the corresponding
 function $(1-\alpha)^{-1}h_Z(\alpha)$ is bounded on [0,1].
For example, we can take functions $Z\in \Z$ such that $Z(t)=0$ for $t\in [\gamma,1]$ with $\gamma\in (0,1)$.
  Then for any $Z\in \V$ we have that
 \[
\lim_{k\to\infty} \int_0^1Z(t) \zeta_k(t) dt=
\lim_{k\to\infty} \int_0^1(1-\alpha)^{-1}h_Z(\alpha)d\mu_k(\alpha)=  \int_0^1(1-\alpha)^{-1}h_Z(\alpha)d\mu (\alpha),
 \]
 and hence
  \begin{equation}\label{der1}
\lim_{k\to\infty} \int_0^1Z(t) \zeta_k(t) dt=\int_0^1Z(t) \zeta (t) dt,
 \end{equation}
with the integral in the right hand side of (\ref{der1}) being finite. This shows that $\cT$ is continuous at $\mu$ and
$\zeta\in \Z^*$, and hence $\mu\in \cP_q$.  \fb
 \end{proof}
  \\

 It follows that if $\cC\subset \cA$ is a weakly$^*$ closed  set of right side continuous monotonically nondecreasing  functions, then $\cT^{-1}(\cC)$ is a closed subset of $\cP$.

 \begin{definition}
 \label{defminim}
 We say that a Kusuoka set $\cM\subset\cP$ is {\rm minimal} if $\cM$ is closed and for any closed Kusuoka set $\cM'$ it holds that $\cM\subset \cM'$. We say that a generating set $\Upsilon$ is {\rm minimal} if $\Upsilon$ is weakly$^*$ closed and for any weakly$^*$ closed generating set $\Upsilon'$ it holds that $\Upsilon \subset \Upsilon'$.
 \end{definition}

 Note that it follows from the above definition that {\em if   minimal Kusuoka set exists, then it is unique.}

 By Propositions \ref{pr-gen} and  \ref{pr-con} we have that a Kusuoka set $\cM$ is minimal iff the set $\cT(\cM)$ is a minimal generating set. Together with Theorem  \ref{th-dens} and  Proposition \ref{pr-rep} this   implies  the following result.

\begin{theorem}
\label{th-uniq}
The set $\cM:=\cT^{-1}(\cU)$ is a Kusuoka set and its closure $\overline\cM=\overline{\cT^{-1}(\cU)}=\cT^{-1}(\bar\cU)$   is the  minimal Kusuoka set.
\end{theorem}

This shows that the minimal Kusuoka set does exist, and as it was mentioned before, by the definition is unique.
 In particular we obtain the following result (cf.~\cite{mor12}).

\begin{cor}
\label{cor:spec}
Let $\rho_\sigma:\Z\to\bbr$ be a spectral measure with spectral function $\sigma\in \Z^*\cap\cS$, i.e.,
\begin{equation}\label{spectr}
\rho_\sigma(Z)=\int_0^1 \sigma(t)F^{-1}_Z(t)dt
\end{equation}
 Then its  minimal Kusuoka set is   given by  the singleton  $\{\cT^{-1}\sigma\}$.
\end{cor}

\section{Kusuoka representation on general spaces}
\label{sec:general}
\setcounter{equation}{0}

So far we assumed that the probability space is nonatomic and used the standard probability space $(\O,\F,P)$.
Since $\rho$ is law invariant, it can be viewed as a function of the corresponding cdf $F(z)=P(Z\le z)$, and we can write $\rho(F)$ rather than $\rho(Z)$. The condition that $\rho$ is defined on the space $\Z=L_p(\O,\F,P)$ means that the space of cdfs on which $\rho$ is defined is such that $F^{-1}\in \Z$.   Note that
\begin{equation}\label{genkus1}
\bbe_F|Z|^p=\int_{-\infty}^{+\infty}|z|^p dF(z)=
\int_0^1 |F^{-1}(t)|^pdt,
\end{equation}
and hence $F^{-1}\in \Z$ means that $\bbe_F|Z|^p<\infty$.

Now let $(\hat{\O},\hat{\F},\hat{P})$ be a general probability space, not necessarily nonatomic. We still can talk about law invariant coherent risk measures $\varrho:\hat{\Z}\to\bbr$ defined on $\hat{\Z}:=L_p(\hat{\O},\hat{\F},\hat{P})$. Again we can view $\varrho(F)$ as a function of the cdf $F$. Note, however, that depending on the choice of $(\hat{\O},\hat{\F},\hat{P})$, the space of allowable cdfs can be different. Suppose, for example, that the space $\hat{\O}=\{\w_1,...,\w_n\}$ is finite with respective probabilities $p_i>0$, $i=1,...,n$. Then the cdf $F$  of a random variable $Z:\hat{\O}\to\bbr$ can be any piecewise constant cdf  with $n$ jumps of sizes $p_i$ at arbitrary points $z_1,...,z_n$. The corresponding $F^{-1}$ is a step function
\begin{equation}\label{genkus2}
F^{-1} =\sum_{i=1}^n z_{\pi(i)}\ind_{I_{\pi(i)}},
\end{equation}
where  $z_{\pi(1)}\leq\cdots\leq z_{\pi(n)}$ are numbers $z_i$ arranged in the increasing order and $\{I_{\pi(1)},...,I_{\pi(n)}\}$ is the partition of the interval [0,1] into    intervals of the respective  lengths $p_{\pi(1)},...,p_{\pi(n)}$; in order for $F^{-1}(\cdot)$ to be right side continuous these should be half open intervals of the form  $I_i=[a_i,b_i)$.

To any law invariant coherent risk measure $\rho:\Z\to\bbr$ we can correspond  its restriction $\hat{\rho}(F)$ to the class of cdfs $F=F_Z$  associated with random variables $Z\in\hat{\Z}$. For example, for finite $\hat{\O}=\{\w_1,...,\w_n\}$ these are cdfs having inverse of
the form (\ref{genkus2}). Of course, the space of (monotonically nondecreasing) step functions of the form (\ref{genkus2}) is much smaller than the space of (monotonically nondecreasing) $p$-power integrable functions on the interval [0,1]. Therefore the class of law invariant coherent risk measures $\varrho:\hat{\Z}\to\bbr$ can be larger than the class of law invariant coherent risk measures $\hat{\rho}$.

\begin{definition}
\label{def:restr}
Let $(\O,\F,P)$ be the standard probability space,
$\Z:=L_p(\O,\F,P)$, $p\in [1,\infty)$,  and $\hat{\Z}:=L_p(\hat{\O},\hat{\F},\hat{P})$ for a general probability space $(\hat{\O},\hat{\F},\hat{P})$. We say that a law invariant coherent risk measure $\varrho:\hat{\Z}\to \bbr$ is {\em  regular} if there exists a law invariant coherent risk measure $\rho:\Z\to\bbr$ such that $\varrho=\hat{\rho}$, where $\hat{\rho}$ is the restriction of $\rho$ to $\hat{\Z}$ in the sense explained above.
Sometimes we use the term \emph{$p$-regular} risk measure to emphasize choice of the space $\Z:=L_p(\O,\F,P)$.
\end{definition}

By (\ref{genrepr})
for a  regular  law invariant coherent risk measure $\varrho=\hat{\rho}$ we can write for $F=F_Z$, $Z\in \hat{\Z}$,  the dual representation
\begin{equation}\label{dualred}
 \varrho(F)=\sup_{\sigma\in \Upsilon}\int_0^1 \sigma(t) F^{-1}_Z(t)dt,
\end{equation}
where $\Upsilon$ is a generating set of the corresponding  risk measure $\rho:\Z\to\bbr$,
and the respective   Kusuoka representation
\begin{equation}\label{genkus3}
 \varrho(F)=\sup_{\mu\in \cM}\int_0^1 \avr_\alpha (F)d\mu(\alpha),
\end{equation}
where $\cM$ is the set of probability measures corresponding to the risk measure $\rho:\Z\to\bbr$. Conversely, if the representation (\ref{dualred}) or
the Kusuoka  representation (\ref{genkus3}) holds and the  right hand side of (\ref{dualred}) (of (\ref{genkus3})) is well defined and  finite for every $F=F_Z$, $Z\in \Z$, then $\varrho$ is a regular  law invariant coherent risk measure. Hence we have the following.

\begin{proposition}
\label{pr-reg}
Law invariant coherent  risk measure  $\varrho:\hat{\Z}\to\bbr$ is $p$-regular iff there exists a set $\Upsilon\subset L_q(\O,\F,P)$ of spectral functions such that the representation {\rm (\ref{dualred})} holds, or equivalently iff  the Kusuoka representation {\rm (\ref{genkus3})} holds.
\end{proposition}

 In Pflug and  R\"omisch \cite[p.63]{pfl07} is given an example of a law invariant coherent risk measure defined on   space of random variables $Z:\hat{\O}\to \bbr$, with   $\hat{\O}=\{\w_1,\w_2\}$ having two elements with unequal probabilities $p_1\ne p_2$, for which the Kusuoka representation does not hold. Thus this  gives an example of a nonregular risk measure (see also Example \ref{ex-fin} below).

\begin{example}
\label{ex-cons}
{\rm
Let us consider the following construction.
Let $(\O,\F,P)$ be the standard probability space and   $\Z:=L_p(\O,\F,P)$ for some $p\in [0,1)$. Let    $\G$ be a sigma subalgebra of $\F$ and $\hat{\Z}=L_p(\O,\G,P)$ be the space of $\G$-measurable random variables $Y\in \Z$. For example, consider a countable  probability space $\hat{\O}=\{\w_1,...\}$ with respective probabilities $p_i>0$, $i=1,...$,  $\sum p_i=1$. We can embed the space of random variables $Y:\hat{\O}\to \bbr$ into $\Z$  by using the following construction. Consider the  partition of the interval [0,1) into intervals  $I_1=[0,p_1)$, $I_2=[p_1,p_1+p_2), ...$, of respective lengths $p_1,p_2,...$.
Let $\G$ be the subalgebra of $\F$ generated by sets $I_1,I_2,...$. Then the corresponding space $\hat{\Z}$ is formed by random variables $Y:[0,1)\to\bbr$ being piecewise constant on each $I_i$, and $\hat{\Z}$ can be identified with the space of random variables on $\hat{\O}$.

We have here  that a law invariant coherent risk measure $\vr:\hat{\Z}\to\bbr$ is   regular  if there exists a law invariant coherent risk measure $\rho:\Z\to\bbr$ such that\footnote{By $\rho  |_{\hat{\Z}}$ we denote the restriction of $\rho$ to $\hat{\Z}$.}
$\rho  |_{\hat{\Z}}=\vr$.
Let $\cF$ be the group of measure preserving transformations of the standard  space (see definition \ref{def-group}).
Consider  the following  subgroup of $\cF$:
\[
\cG:=\{T\in \cF:  T(A)\in \G,\; \forall A\in \G\}.
\]
We have that  $\eta_1,\eta_2\in \hat{\Z}^*$ are distributionally equivalent if there exists $T\in \cG$ such that $\eta_1=\eta_2\circ T$. As it was pointed in Remark \ref{rem-de}, two random variables defined on a {\em nonatomic} space are distributionally equivalent if and only if  one can be transformed into the other by a measure preserving transformation. Clearly  the ``if" part is true for spaces with atoms as well.

Consider the following condition.
\begin{itemize}
\item[(A)]
For  any $\eta\in \hat{\Z}^*$, $\eta\ge 0$,   there exists  monotonically nondecreasing right hand side continuous function $\zeta : [0,1)\to \bbr$  and $T\in\cG$  such that $\eta=\zeta\circ T$ a.e.
\end{itemize}

\begin{theorem}
\label{th-reg}
If    condition  {\rm (A)} holds, then every law invariant coherent risk measure $\vr:\hat{\Z}\to\bbr$ is regular.
\end{theorem}

\begin{proof}
Let $\vr:\hat{\Z}\to\bbr$ be a law invariant coherent risk measure. It has the dual representation
\begin{equation}\label{lin1}
\vr(Y)=\sup_{\eta\in \hat{\cA}} \int \eta(\w) Y(\w)dP(\w),
\end{equation}
where $\hat{\cA}\subset \hat{\Z}^*$ is a set of density  functions.  Since $\vr$ is law invariant, the set $\hat{\cA}$  is invariant
 with respect to transformations of the group $\cG$ (see, e.g., proof of  Theorem 3.2 in \cite{sha:2012}).

Suppose that condition (A) holds.
Let $Y\in \hat{\Z}$ be monotonically nondecreasing function and
$\hat{\Upsilon}$ be the subset of $\hat{\cA}$ formed by monotonically nondecreasing $\eta\in \hat{\cA}$. Since
$\hat{\cA}$  is invariant with respect to transformations of the group $\cG$ it follows by property (A) that it suffices to take the maximum in (\ref{lin1}) with respect to   $\eta\in \hat{\Upsilon}$. Define now
\begin{equation}\label{lin2}
\rho(Z)=\sup_{\eta\in \hat{\Upsilon}} \int \eta(\w) F^{-1}_Z(\w)dP(\w).
\end{equation}
We have that $\rho$ is a law invariant coherent risk measure and $\rho  |_{\hat{\Z}}=\vr$.  \fb
 \end{proof}
}
\end{example}

If the space  $\hat{\O}=\{\w_1,...,\w_n\}$ is finite, equipped with equal probabilities, then the corresponding group $\cG$ of measure preserving transformations is given by the set of permutations of the set $\hat{\O}$. It follows that condition (A) holds and hence by Theorem \ref{th-reg} we have the following:

\begin{cor}
\label{cor-finite}
Let $\hat{\O}=\{\w_1,...,\w_n\}$ be finite, equipped  with equal probabilities $p_i=1/n$, $i=1,...,n$. Then every law invariant coherent risk measure $\vr:\hat{\Z}\to\bbr$  is regular.
\end{cor}

In the setting of   Corollary \ref{cor-finite},
$F_Z^{-1}(\cdot)$, $Z\in \hat{\Z}$, are piecewise constant (step) functions. Thus the spectral function $\sigma(\cdot)$ can be changed on the respective  intervals $I_i$ (of length $1/n$)  without changing the corresponding integral.  For example, the spectral functions $\sigma(\cdot)$ can be taken to be step functions  constant on the intervals $I_i$.  Consequently, although every  law invariant coherent risk measure $\vr:\hat{\Z}\to\bbr$  is regular here, the corresponding (lifted)  risk measure $\rho:\Z\to\bbr$ is not unique.

\begin{example}
\label{ex-fin}
{\rm
Let $(\hat{\O},\hat{\F},\hat{P})$ be a probability space with atoms. Consider the setting of Example~\ref{ex-cons} and assume that $(\hat{\O},\hat{\F},\hat{P})$ can be identified with the corresponding embedded space $(\O,\G,P)$ for some sigma algebra $\G\subset \F$,  and consider the   space $\hat{\Z}=L_p(\O,\G,P)$. For some $\hat{p}>0$, let
 $\hat{A}\in \hat{\F}$ be a (nonempty)  set of all atoms having the same  probability $\hat{p}$. Clearly the set $\hat{A}$ is finite,  let $k:=|\hat{A}|$ be cardinality of $\hat{A}$ (the set $\hat{A}$ could be a singleton, i.e., it could be that $k=1$). Let $A\in \G$ be the corresponding subset of $\O=[0,1]$ (given by union of the respective subintervals of $[0,1]$). Assume that  $\hat{\O}$ is not a finite set with equal probabilities (as in Corollary \ref{cor-finite}), so
 that   $\hat{P}(\O\setminus A)>0$.
Define
\begin{equation}\label{nonreg}
 \vr(Y):=\frac{1}{k}\sum_{\w\in \hat{A}}  Y(\w).
\end{equation}
Clearly $\vr:\hat{Z}\to \bbr$ is a coherent risk measure.
Suppose further  that the following condition holds.
\begin{itemize}
\item[(B)]
If $Y,Y'\in \hat{Z}$ are distributionally equivalent, then there exists a permutation $\pi$ of the set $\hat{A}$ such that $Y'(\w)=Y(\pi(\w))$ for any $\w\in \hat{A}$.
\end{itemize}
Under this condition, the  risk  measure $\vr$ is   law invariant as well. Let us show that it is not regular.

 Indeed, let us argue by a contradiction.   Suppose that
 $\vr=\hat{\rho}$ for some  law invariant coherent risk measure $\rho:\Z\to \bbr$ and  $\hat{\rho}:=\rho  |_{\hat{\Z}}$. Note that
  $\vr(Y)$ depends only on values of $Y(\cdot)$ on the set $\hat{A}$, i.e.,
 for any $Y,Y'\in\hat{\Z}$ we have that  $\rho (Y)= \rho (Y')$ if $Y(\w)=Y'(\w)$ for all $\w\in A$. In particular, if $Y(\w)=0$ for all $\w\in A$, then  $\rho (Y)=0$.
 It follows that $\rho  (\ind_{B})=0$ for any $\G$-measurable  set $B\subset \O\setminus A$. Note that the set $\O\setminus A$ has a nonempty interior since the set $A$ is a union of a finite number of intervals,  and hence we can take $B\subset \O\setminus A$ to be an open interval. 
Since $\rho$ is law invariant it follows that  $\rho  (\ind_{T(B)})=0$ for any $T\in \cF$. Hence there is a family of sets $B_i\in \F$, $i=1,...,m$, such that $\rho  (\ind_{B_i})=0$, $i=1,...,m$, and $[0,1]=\cup_{i=1}^m B_i$. It follows that for any bounded  $Z\in\Z$ there are   $c_i\in\bbr$, $i=1,...,m$, such that $Z\le \sum_{i=1}^m c_i \ind_{B_i}$. Consequently
 $
 \rho(Z) \le \sum_{i=1}^m c_i \rho(\ind_{B_i})=0,
 $
 this clearly is a contradiction.

It follows that the risk measure $\vr$, defined in (\ref{nonreg}),  does not have a Kusuoka representation. Note that it was only essential in the above construction that the restriction of the risk measure  $\vr:\hat{Z}\to\bbr$ to the set $\hat{A}$ is a law invariant coherent risk measure. So, for example, under the above assumptions  the risk measure
$\vr(Y):=\max\{Y(\w):\w\in \hat{A}\}$ is also not regular.

As far as condition (B) is concerned, suppose for example that
the set $\hat{\O}=\{\w_1,...,\w_n\}$ is finite, equipped with respective probabilities $p_1,...,p_n$. For some $\hat{p}\in \{p_1,...,p_n\}$ let $K\subset \{1,...,n\}$ be the index set such that $p_i=\hat{p}$ for all $i\in K$.   Suppose that
\begin{equation}\label{nonrcon}
 \sum_{i\in I}p_i\ne \sum_{j\in J}p_j,\;\;\forall 
 I\subset K,\;\forall J\subset\{1,...,n\}\setminus K. 
\end{equation}
 Then  condition (B) holds for the set   $\hat{A}:=\{\w_i:i\in K\}$. That is, condition (\ref{nonrcon}) ensures existence of a nonregular risk measure $\vr$. 
}
\end{example}

\section{Maximality of Kusuoka sets}
\label{sec;maxim}
\setcounter{equation}{0}

In this section we discuss Kusuoka representations with respect to stochastic dominance relations.
As before we consider probability measures $\mu$ supported on the interval  $[0,1]$ and continue identifying the measure $\mu$ with its cumulative distribution function $\mu(\alpha)= \mu\{ (-\infty,\alpha]\}$, $\alpha\in \bbr$. Recall that since $\mu$ is supported on [0,1], it follows that $\mu(\alpha)=0$ for $\alpha<0$ and $\mu(\alpha)=1$ for $\alpha\ge 1$.

\begin{definition}
\label{def:firstord}
It is said that $\mu_1$ is {\em dominated in first stochastic order} by $\mu_2$, denoted $\mu_{1}\fo\mu_{2}$, if  $\mu_{1}\left(\alpha\right)\ge\mu_{2}\left(\alpha\right)$
for all $\alpha\in\bbr$.
If moreover, $\mu_1\ne \mu$ we write $\mu_{1}\fos\mu_{2}$.
\end{definition}

A set of measures $(\cM, \fo)$, equipped with the first order stochastic dominance relation, is a partially ordered set.

\medskip
In the space of functions $\sigma\in \Z^*$ we consider the following dominance relation.
\begin{definition}
\label{def:secord}
For $\sigma_1,\sigma_2\in\Z^*$  it is said that $\sigma_1$ is {\em majorized} by $\sigma_2$, denoted $\sigma_1\so\sigma_2$,  if
\begin{eqnarray}
\label{maj1}
&\int_\gamma^1 \sigma_1(t)dt\le \int_\gamma^1 \sigma_2(t)dt\;\; \text{ for all } \gamma\in [0,1], \text{  and}\\
\label{maj2}
&  \int_0^1 \sigma_1(t)dt= \int_0^1 \sigma_2(t)dt.
\end{eqnarray}
\end{definition}

 \begin{remark}
\label{rem-1}
 {\rm
For monotonically nondecreasing functions $\sigma_1,\sigma_2:[0,1]\to\bbr$ the above concept of majorization is closely related to the concept of  dominance in the convex order. That is, if $\sigma_1$ and $\sigma_2$
are monotonically nondecreasing   right  side continuous  functions, then they can be viewed as right side quantile functions $\sigma=F^{-1}_{Z_1}$ and $\sigma_2=F_{Z_2}^{-1}$ of some respective random variables $Z_1$ and $Z_2$. It is said  that $Z_1$ dominates $Z_2$ in the convex order if $\bbe[u(Z_1)]\ge \bbe[u(Z_2)]$ for all   convex functions $u:\bbr\to\bbr$ such that the expectations exist. Equivalently this can written as (see, e.g., M\"uller and   Stoyan \cite{mulst})
\[
\int_\gamma^1 F^{-1}_{Z_1}(t)dt\ge \int_\gamma^1 F^{-1}_{Z_2}(t)dt,\;\forall \gamma\in [0,1];\;{\rm and} \;
\bbe[Z_1]= \bbe[Z_2].
\]
The  dominance in the convex (concave)  order was used in studying risk measures  in F\"ollmer and  Schield \cite{fol04}  and Dana \cite{dana}, for example.
}
\end{remark}

 Note that if $\sigma_1,\sigma_2\in \Z^*\cap\cS$ and $\sigma_1\so\sigma_2$, then
 \begin{equation}\label{soineq}
\int_0^1 \sigma_1(t)F^{-1}_Z(t)dt\le \int_0^1 \sigma_2(t)F^{-1}_Z(t)dt, \;\;Z\in \Z.
 \end{equation}

\begin{example}
{\rm
Let $\rho$ be given by maximum of a finite number of spectral risk measures, i.e.,
\begin{equation}\label{maxspec}
\rho(Z):=\max_{1\leq i\le n} \int_0^1 \sigma_i(t)F^{-1}_Z(t)dt,\;\;Z\in \Z,
\end{equation}
for some $\sigma_i\in \Z^*\cap\cS$, $i=1,...,n$.
Then the set $\Upsilon:=\{\sigma_1,...,\sigma_n\}$ is
 a generating set and the convex hull of $\OO(\Upsilon)$ is the dual set of $\rho$.  An element $\sigma_i$ of $\{\sigma_1,...,\sigma_n\}$ is an exposed point of the dual set iff  $\sigma_i$ can be strongly separated from the other $\sigma_j$, i.e., iff there exists $Z\in \Z$ such that
$
 \lan  \sigma_i,Z\ran \ge \lan  \sigma_j,Z\ran+\e,
$
 for some $\e>0$ and all $j\ne i$.  So the generating  set $\Upsilon$ is minimal   iff every $\sigma_i$ can be strongly separated from the other $\sigma_j$.

 If there is $\sigma_i$ which is majorized by another $\sigma_j$, then it follows  by (\ref{soineq}) that removing $\sigma_i$ from the set $\Upsilon$ will not change the corresponding maximum, and
 hence $\Upsilon\setminus\{\sigma_i\}$ is still a generating set. Therefore the condition:  ``{\em every $\sigma_i$  is not majorized by any other $\sigma_j$}" is necessary for the set $\Upsilon$ to be minimal. For $n=2$ this condition is also sufficient. However,  it is not sufficient already  for $n=3$.
 For example let $\sigma_1$ and $\sigma_2$ be such that
 $\sigma_1$ is not majorized by  $\sigma_2$,  and $\sigma_2$ is not majorized by  $\sigma_1$, and let $\sigma_3:=(\sigma_1+\sigma_2)/2$. Then clearly $\sigma_3$ can be removed from $\Upsilon$, while $\sigma_3$ is not majorized by $\sigma_1$ or $\sigma_2$. Indeed, if $\sigma_3$ is majorized say by $\sigma_1$, then
 \[
 \int_\gamma^1 \sigma_3 (t)dt=\half\left(\int_\gamma^1 \sigma_1(t)dt+\int_\gamma^1 \sigma_2(t)dt\right)
 \le \int_\gamma^1 \sigma_1(t)dt,\;\gamma\in [0,1],
 \]
and hence
\[
 \int_\gamma^1 \sigma_2(t)dt\le \int_\gamma^1 \sigma_1(t)dt,\;\gamma\in [0,1],
 \]
a contradiction with the condition that $\sigma_2$ is not majorized by $\sigma_1$.
}
\end{example}

We show now that first order dominance of measures is transformed by  the operator $\cT$ into majorization order in the sense of Definition \ref{def:secord}.
The converse implication, $$\{\sigma_1\so\,\sigma_2\}\Rightarrow\{\cT^{-1}
\sigma_1\fo\,\cT^{-1}\sigma_2\}$$  does not hold in general. This shows  that the first order stochastic dominance indeed is a stronger concept. A simple counterexample is provided by the measures $\mu_1:=0.1 \delta_{0.2}+0.9 \delta_{0.6}$ and $\mu_2:=0.2 \delta_{0.5}+0.8 \delta_{0.9}$.

\begin{proposition}
\label{pr:orders}
For  $\mu_1,\mu_2\in \cP_q$ it holds that
if $\mu_1\fo\mu_2$, then $\cT\mu_1\so \cT\mu_2$.
\end{proposition}

\begin{proof}
By definition of $\cT$ and reversing the order of integration we can write for $\gamma\in (0,1)$,
\begin{align*}
\int_{0}^{\gamma}(\cT\mu_1)(t)d t  =\int_{0}^{\gamma}\int_{0}^{t}\frac{1}{1-\alpha}
d\mu_1(\alpha)\,dt =\int_{0}^{\gamma}\int_{\alpha}^{\gamma}
\frac{1}{1-\alpha}dt\,d\mu_1(\alpha) =\int_{0}^{\gamma}\frac{\gamma-\alpha}{1-\alpha}d\mu_1(\alpha).
\end{align*}
Now by Riemann-Stieltjes integration by parts
\begin{align*}
\int_{0}^{\gamma}(\cT\mu_1)(t)dt  =\left.\frac{\gamma-\alpha}{1-\alpha}\mu_1(\alpha)
\right|_{\alpha=0}^{\gamma}-\int_{0}^{\gamma}\mu_1(\alpha)
d\left(\frac{\gamma-\alpha}{1-\alpha}\right)=
\int_{0}^{\gamma}\mu_1(\alpha)\frac{1-\gamma}
{\left(1-\alpha\right)^{2}}d\alpha,
\end{align*}
where we used that $\mu_1(0^-)=\mu_2(0^-)=0$.
Since  $\mu_1\fo\mu_2$ we have  that $\mu_1(\cdot)\ge \mu_2(\cdot)$ and hence
\begin{equation*}
\int_0^\gamma(\cT\mu_1)(t)dt=
\int_{0}^{\gamma}\mu_1(\alpha)\frac{1-\gamma}
{\left(1-\alpha\right)^{2}}d \alpha \ge  \int_{0}^{\gamma}\mu_2(\alpha)
\frac{1-\gamma}{\left(1-\alpha\right)^{2}}d\alpha
=\int_0^\gamma(\cT\mu_2)(t)dt.
\end{equation*}
Because of $\int_0^1 (\cT\mu)(t)dt=1$, this implies that
\begin{equation*}
\int_\gamma^1(\cT \mu_1)(t)dt\le \int_\gamma^1(\cT\mu_2)(t)dt,
\end{equation*}
and hence  $\cT \mu_1 \so\, \cT \mu_2$.
\fb\end{proof}
\\

Let $\cM$ be a Kusuoka set and $\mu_1,\mu_2\in \cM$.
As it was pointed above, it follows from (\ref{soineq}) that if
 $\cT \mu_1 \so\, \cT \mu_2$, then measure $\mu_1$ can be removed from $\cM$. Hence it follows by Proposition~\ref{pr:orders} that if  $\mu_1\fo\,\mu_2$, then the measure $\mu_1$ can be removed from $\cM$.

\begin{definition}\label{def-domin}
A measure $\mu\in\cM\subset \cP$ is called a {\em maximal element}, a {\em maximal measure} or {\em non-dominated measure} of $(\cM,\fo)$, if there is no $\nu\in \cM$ satisfying $\mu\fos\nu$.
\end{definition}

The measures $\delta_0$ and $\delta_1$ are extremal in the sense that
\[\delta_0\fo\mu\fo\delta_1\qquad\text{for all }\mu\in\cP,\]
that is $\delta_1$ ($\delta_0$, resp.) is always a maximal (minimal, resp.) measure.

The next theorem elaborates that it is sufficient to consider the non-dominated measures within the closure of any Kusuoka set.
\begin{theorem}
\label{th:extmes}
Let $\cM$ be a Kusuoka set of law invariant coherent risk measure $\rho:\Z\to\bbr$, and $\overline\cM$ be its topological closure. Then the augmented set $\{\mu^\prime\in\cP\colon \mu^\prime\preccurlyeq\mu\;\text{\rm  for some }\mu\in\cM\}$ is a Kusuoka set as well. Furthermore, the set of extremal measures of $\overline\cM$, i.e.,  $\cM^\prime:=\{\mu\in\overline\cM: \nu  \not\succ\mu\,\text{ for all }\nu\in\cM\}$ is a Kusuoka set.
\end{theorem}

\begin{proof}
Let $\mu^\prime\fo\mu\in\cM$.
As $\alpha\mapsto\AVaR_\alpha(Y)$ is a nondecreasing function and as $\mu(\cdot)\le\mu^\prime(\cdot)$ it follows by Riemann-Stieltjes integration by parts that
\begin{align*}
\int_{0}^{1}\AVaR_{\alpha}(Y)d\mu(\alpha)
 & =\left.\mu(\alpha)\AVaR_{\alpha}(Y)\right|_{\alpha=0_-}^1-\int_0^1\mu(\alpha)d\AVaR_{\alpha}(Y)\\
 & \ge\left.\mu^\prime(\alpha)\AVaR_{\alpha}(Y)\right|_{\alpha=0_-}^{1}-\int_0^1\mu^\prime(\alpha)d\AVaR_\alpha(Y)\\
 & =\int_0^1\AVaR_{\alpha}(Y)\mu^\prime(d\alpha),
\end{align*}
which is the first assertion.
\medskip

For the second assertion recall that $(\cM,\fo)$ is a partially ordered set and so is $(\overline\cM,\fo)$.
Consider a chain $\cC\subset\overline\cM$, that is for every $\mu$, $\nu$ it holds that $\mu\fo\nu$ or $\nu\fo\mu$ (totality).
Then the chain $\cC$ has an upper bound in $\overline\cC\subset\overline{\cM}$: to accept this (cf. the proof of Helly's Lemma in~\cite{vdVaart}) define $$\mu_\cC\left(x\right):=\inf_{\mu\in\cC}\mu(x)$$
(the upper bound), which is a positive, non-decreasing function satisfying
$\mu_\cC(1)=1$. As any $\mu\in\cC$ is right side continuous it is upper semi-continuous, thus $\mu_\cC$, as an infimum, is upper semi-continuous as well, hence $\mu_\cC$ is right side continuous and $\mu_\cC$ thus represents a measure, $\mu_\cC\in\cP$.
To show that $\mu_\cC\in\overline\cC$ let $x_i$ be a dense sequence in $[0,1]$ and choose $\mu_{i,n}\in\cC$ such that $\mu_{i,n}(x_i)<\mu_\cC(x_i)+2^{-n}$. As $\cC$ is a chain one may define $\mu_n:=\min\{\mu_{i,n}: i=1,2,\dots n\}$. It holds that $\mu_n(x_i)<\mu_\cC(x_i)+2^{-n}$ for all $i=1,2,\dots n$. As $x_i$ is dense, and as $\mu_n$, as well as $\mu_\cC$ are right side continuous it follows that $\mu_n\to\mu_\cC$ uniformly, hence $\mu_\cC\in\overline\cC$.

By Zorn's Lemma there is at least one maximal element $\mu^{*}$ in
$\overline{\cM}$, that is there is no element $\nu\in\overline{\cM}$
such that $\nu\succ\mu^*$. Hence
\begin{eqnarray*}
\mu^{*} & \in & \left\{ \mu\in\overline{\cM}:\;\neg\exists\nu\in\cM:\nu\succ\mu\right\} \\
 & = & \left\{ \mu\in\overline{\cM}:\;\forall\nu\in\cM:\nu\not\succ\mu\right\} =\cM^\prime
\end{eqnarray*}
and $\cM^\prime$ thus is a non-empty Kusuoka set.
Recall that $\cM^{\prime}\subset\overline{\cM}$,
hence
\begin{eqnarray*}
\rho\left(Y\right) & = & \sup_{\mu\in\cM}\int_{0}^{1}\AVaR_{\alpha}\left(Y\right)\mu\left(\mathrm{d}\alpha\right)\\
 & = & \sup_{\mu\in\overline{\cM}}\int_{0}^{1}\AVaR_{\alpha}\left(Y\right)\mu\left(\mathrm{d}\alpha\right) \ge  \sup_{\mu\in\cM^{\prime}}\int_{0}^{1}\AVaR_{\alpha}\left(Y\right)\mu\left(\mathrm{d}\alpha\right).
\end{eqnarray*}

To establish equality assume that $\rho\left(\cdot\right)\neq\sup_{\mu\in\cM^{\prime}}\int_{0}^{1}\AVaR_{\alpha}\left(\cdot\right)\mu\left(\mathrm{d}\alpha\right)$.
Then there is a random variable $Y$ satisfying
\begin{equation}
\rho\left(Y\right)>\sup_{\mu^{\prime}\in\cM^{\prime}}\int_{0}^{1}\AVaR_{\alpha}\left(Y\right)\mu^{\prime}\left(\mathrm{d}\alpha\right).\label{eq:1}
\end{equation}
For this $Y$ and for some $\varepsilon>0$ choose  $\mu\in\overline{\cM}$ such that
$\rho\left(Y\right)< \varepsilon+\int_{0}^{1}\AVaR_{\alpha}\left(Y\right)\mu\left(\mathrm{d}\alpha\right)$.

Consider the cone $\M_{\mu}:=\left\{ \nu\in\overline{\cM}:\mu\fo\nu\right\}$.
Notice that $\left(\cM_{\mu},\fo\right)$ again is a partially ordered set, for which Zorn's Lemma implies that there is a maximal element $\bar{\mu}\in\cM_{\mu}$ with respect to $\fo$, and it holds that $\mu\fo\bar\mu$ by construction.

We claim that $\bar{\mu}\in\cM^{\prime}$. Indeed, if it were
not, then there is $\nu^\prime\in\cM$ with $\nu^\prime\succ\bar{\mu}$.
But this means $\mu\fo\bar\mu\fos\nu^{\prime}$, so $\nu^{\prime}\in\cM_{\mu}$
and hence $\bar{\mu}\prec\nu^{\prime}$, contradicting the fact that $\bar{\mu}$
is a maximal measure in $\cM_{\mu}$, and hence $\bar{\mu}\in\cM^{\prime}$.

Now by Riemann-Stieltjes integration by parts
\begin{align*}
\rho\left(Y\right)-\varepsilon& <\int_{0}^{1}\AVaR_{\alpha}(Y)\mathrm{d}\mu(\alpha)\\
 & =\left.\mu(\alpha)\AVaR_{\alpha}(Y)\right|_{\alpha=0}^{1}-\int_{0}^{1}\mu(\alpha)\mathrm{d}\AVaR_{\alpha}(Y)\\
 & \le\left.\bar{\mu}(\alpha)\AVaR_{\alpha}(Y)\right|_{\alpha=0}^{1}-\int_{0}^{1}\bar{\mu}(\alpha)\mathrm{d}\AVaR_{\alpha}(Y)\\
 & =\int_{0}^{1}\AVaR_{\alpha}(Y)\bar{\mu}(\mathrm{d}\alpha),
\end{align*}
as $\alpha\mapsto\AVaR_{\alpha}(Y)$ is a non-decreasing function
and as $\mu\fo\bar{\mu}$, that is $\mu(\cdot)\ge\bar\mu(\cdot)$.

But as $\overline{\mu}\in\cM^\prime$ this is a contradiction to \eqref{eq:1}, such that the assertion   holds indeed.
\fb\end{proof}

\section{Examples}
\label{sec:examp}
\setcounter{equation}{0}

In order to demonstrate the minimal Kusuoka representation we have chosen two examples. The first is taken from \cite{DentPenevRusz}.
The second example elaborates the Kusuoka representation of the higher order semideviations.

\subsection{Higher order measures}
This first example generalizes the Average Value-at-Risk as introduced in \eqref{2k}.
\begin{proposition}\label{prop11} The minimal Kusuoka representation of the risk measure
\begin{eqnarray}\label{eq:47}
\rho(Z):=\inf_{\,t\in\mathbb R} \left\{t+c\, \Vert (Z-t)_+\Vert_p\right\},\;Z\in L_p(\O,\F,P),
\end{eqnarray}
($c> 1$ and $1<p<\infty$) is given by
\begin{eqnarray}\label{eq:23}
\rho(Z)=\sup \left\{\rho_\sigma(Z): \, \sigma\in\cS\text{ and }\Vert\sigma\Vert_q= c \right\},
\end{eqnarray}
where $\rho_\sigma$ is the spectral risk measure associated with the spectrum $\sigma$ (cf.~\eqref{spectr}) and $\frac 1 p+\frac 1 q=1$.
\end{proposition}

\begin{rem}
{\rm
Equation
\eqref{eq:23} is not a Kusuoka representation in its original form, but it is a concise way of writing $$\rho_\sigma(Z)=\int_0^1 \sigma(u)F_Z^{-1}(u)du= \int_0^1 \AVaR_\alpha(Z)d\mu(\alpha),$$ where $\mu=\cT^{-1}\sigma$ according to \eqref{3k} and \eqref{spectr}.
}
\end{rem}
\begin{rem}
{\rm
It is obvious that \eqref{eq:23} represents the Average Value-at-Risk, whenever $c=\frac 1{1-\alpha}$ and $p=1$, such that Proposition~\ref{prop11} generalizes the Average Value-at-Risk.
}
\end{rem}
\begin{proof}
The Kusuoka representation of \eqref{eq:47} is  given in \cite{DentPenevRusz} in the form
\begin{eqnarray}\label{dent}
\rho(Z)=\sup \left\{\rho_\sigma(Z): \, \sigma\in\cS \text{ and }\Vert\sigma\Vert_q\le c \right\},
\end{eqnarray}
i.e., the dual set  of $\rho$ is $\cA=\OO\left\{\sigma\in\cS: \Vert\sigma\Vert_q\le c \right\}$ (of course, the condition $\Vert\sigma\Vert_q\le c$ implies that the set $\cA$ is a subset of the dual space $L_q(\O,\F,P)$).
We prove first that it is enough to consider functions $\sigma$ satisfying $\Vert\sigma\Vert_q=c$. Indeed, for $\sigma$ satisfying $\Vert\sigma\Vert_q<c$ define $\sigma_\Delta(u):=\ind_{[\Delta,1]}(u)\cdot\left(\sigma(u)+ \frac 1{1-\Delta}\int_0^\Delta \sigma(u^\prime)du^\prime\right)$ ($\Delta\in[0,1)$). It is evident that $\Delta\mapsto\Vert\sigma_\Delta\Vert_q$ is a continuous and unbounded function, hence there is a $\Delta_0$ such that $\Vert\sigma_{\Delta_0}\Vert=c$. Moreover $\sigma\so \sigma_\Delta$ by construction, such that one may pass~-- according to \eqref{soineq}~-- to $\sigma_{\Delta_0}$ without changing the objective in \eqref{dent} to get the representation (\ref{eq:23}).

The set in the dual corresponding to \eqref{eq:23} is
$$\cC=\OO\left(\left\{\zeta\in \cS:  \, \Vert\zeta\Vert_q= c\right\}\right).$$

Now let $\zeta\in\cC\subset\cA$ be chosen, and consider  the random variable $Z:=\zeta^{q-1}$.
The function  $g_Z(\cdot):= \lan \cdot, Z \ran$ (defined in Definition~\ref{def:expos}) attains its maximum over $\cA$ at $\zeta$, because $$g_Z\left(\zeta^\prime\right)\le \Vert\zeta^\prime\Vert_q\cdot\Vert Z\Vert_p\le c \Vert Z\Vert_p=c \Vert\zeta\Vert_q^{p/q}= c^q=\Vert\zeta\Vert_q^q= g_Z(\zeta)$$  by H\"older's inequality whenever $\zeta^\prime\in\cA$. Moreover $\zeta$ is the unique maximizer, as $1<p<\infty$. This shows, using the notation from Definition~\ref{def:expos}, that $\cC$ collects just maximizers, that is $\cC=\ex(\cA)$. This completes  the proof.
\fb\end{proof}

\subsection{Higher order semideviation}

Our second example addresses the $p-$semideviation risk measure. We elaborate its Kusuoka representation, as well as the minimal Kusuoka representation.
\begin{proposition}
For $p\ge1$ the Kusuoka representation of the $p-$semideviation
\[
\rho\left(Z\right):=\mathbb{E}[Z]+\lambda \left\Vert \left(Z-\mathbb{E}Z\right)_{+}\right\Vert _p,\;Z\in L_p(\O,\F,P),
\]
($0\le\lambda\le 1$) is
\begin{equation}\label{eq:3}
\rho\left(Z\right)=\sup_{\sigma\in\cS}\left(1-\tfrac\lambda{\left\Vert \sigma\right\Vert _q}\right)\mathbb{E}[Z]+\tfrac\lambda{\left\Vert \sigma\right\Vert _{q}}\rho_\sigma\left(Z\right).
\end{equation}
The representation \eqref{eq:3} is moreover the  {\em minimal} Kusuoka representation whenever $p>1$.
\end{proposition}

\begin{rem}
{\rm
As in the first example the formula presented in (\ref{eq:3}) does not appear as a Kusuoka representation in its traditional form. To make the Kusuoka representation evident we rewrite it as
\begin{equation}\label{eq:18}
\rho(Z)= \sup_{\mu \in \cP_q}\left\{\left(1-\tfrac\lambda{\left\Vert \cT\mu\right\Vert _q}\right)\mathbb{E}[Z]+\tfrac\lambda{\left\Vert \cT\mu\right\Vert _q}\int_0^1 \AVaR_\alpha(Z)d\mu(\alpha)\right\},
\end{equation}
such that the supremum in the representation (\ref{eq:18})  is among all measures of the form
\begin{eqnarray*}
\left(1-\tfrac\lambda{\left\Vert \cT\mu\right\Vert _q}\right)\delta_0+\tfrac\lambda{\left\Vert \cT\mu\right\Vert _q}\mu,\;\; \mu\in\cP_q,
\end{eqnarray*}
which finally provides a Kusuoka representation.
}
\end{rem}

\begin{rem}
{\rm
Notice that $\left\Vert \sigma\right\Vert _{q}\ge\left\Vert \sigma\right\Vert _{1}=\int_{0}^{1}\sigma(u)\mathrm{d}u=1$,
such that $1-\frac\lambda{\Vert \sigma \Vert_q}\ge 0$ and \eqref{eq:3} is indeed a positive quantity whenever  $0\le\lambda\le1$.
}
\end{rem}

\begin{proof}
By H\"older's duality $\mathbb{E}[Z\zeta]\le\left(\mathbb{E}
\left|Z\right|^{p}\right)^{\frac{1}{p}}
\left(\mathbb{E}\left|\zeta\right|^{q}\right)^{\frac{1}{q}}$
and $\left(\mathbb{E}\left|Z\right|^{p}\right)^{\frac{1}{p}}=\sup\left\{ \mathbb{E}[Z \zeta]\colon\left\Vert \zeta\right\Vert _{q}\le1\right\} $.
Clearly
\[
\left(\mathbb{E}[Z_{+}^{p}]\right)^{\frac{1}{p}}=\sup\left\{ \mathbb{E}[Z \zeta]\colon \zeta\ge0,\left\Vert \zeta\right\Vert _{q}\le1\right\} ,
\]
the supremum being attained at $\zeta=\frac{Z_{+}^{p-1}}{\left\Vert Z_{+}^{p-1}\right\Vert _{q}}$.
It follows that
\begin{equation}
\left\Vert \left(Z-\mathbb{E}[Z]\right)_{+}\right\Vert _{p}=\sup\left\{ \mathbb{E}[\left(Z-\mathbb{E}[Z]\right)\zeta]\colon \zeta\ge0,\left\Vert \zeta\right\Vert _{q}\le1\right\} .\label{eq:5}
\end{equation}

Now
\begin{eqnarray}
\rho(Z) & = & \mathbb{E}[Z]+\lambda\left\Vert \left(Z-\mathbb{E}[Z]\right)_{+}\right\Vert _{p}\nonumber \\
 & = & \sup\left\{ \mathbb{E}[Z]+\lambda\,\mathbb{E}[\left(Z-\mathbb{E}[Z]\right)
 \zeta]\colon \zeta\ge0,\left\Vert \zeta\right\Vert _{q}\le1\right\} \nonumber \\
 & = & \sup\left\{ \left(1-\lambda\,\mathbb{E}[\zeta]\right)\mathbb{E}[Z]
 +\lambda\,\mathbb{E}[Z \zeta]\colon \zeta\ge0,\left\Vert \zeta\right\Vert _{q}\le1\right\} \label{eq:21} \\
 & = & \sup\left\{ \left(1-\lambda\tfrac{\mathbb{E}[\zeta]}{\left\Vert \zeta\right\Vert _{q}}\right)\mathbb{E}[Z]+\lambda\,\mathbb{E}
 \left[Z\tfrac{\zeta}{\left\Vert \zeta\right\Vert _{q}}\right]\colon \zeta\ge0\right\} \nonumber \\
 & = & \sup\left\{ \left(1-\tfrac\lambda{\left\Vert \zeta\right\Vert _{q}}\right)\mathbb{E}[Z] +\tfrac{\lambda}{\left\Vert \zeta\right\Vert _{q}}\mathbb{E}[Z \zeta]\colon \zeta\ge0,\,\mathbb{E}[\zeta]=1\right\} .\label{eq:4}
\end{eqnarray}

For $\zeta$ feasible in \eqref{eq:4} define the function $\sigma_\zeta\left(u\right):=\VaR_{u}\left(\zeta\right)$
and observe that
\[
\left\Vert \sigma_\zeta\right\Vert _q^q=\int_0^1\sigma_\zeta^q(\alpha)\mathrm{d}\alpha
=\mathbb{E}[\zeta^q]=\left\Vert \zeta\right\Vert_q^q.
\]
Now notice that every random variable $\zeta$ in \eqref{eq:4} is given
by $\zeta=\sigma\left(U\right)$ for a uniform $U$ where $\sigma$ is
non-negative, non-decreasing and $\int_0^1\sigma\left(u\right)\mathrm{d}u=1$
-- that is, $\sigma$ is a spectral function. It follows that
\begin{eqnarray}
\rho\left(Z\right) & = & \sup_{\sigma\in\cS}\sup\left\{ \left(1-\tfrac{\lambda}{\left\Vert \zeta\right\Vert _{q}}\right)\mathbb{E}[Z]+\tfrac{\lambda}{\left\Vert \zeta\right\Vert _{q}}\mathbb{E}[Z \zeta]\colon \zeta\ge0,\,\mathbb{E}[\zeta]=1,\,\VaR\left(\zeta\right)=
\sigma\right\}\nonumber \\
 & = & \sup_{\sigma\in\cS}\sup\left\{ \left(1-\tfrac{\lambda}{\left\Vert \sigma\right\Vert _{q}}\right)\mathbb{E}[Z]+\tfrac{\lambda}{\left\Vert \sigma\right\Vert _{q}}\mathbb{E}[Z \zeta]\colon \zeta\ge0,\,\mathbb{E}[\zeta]=1,\, \VaR\left(\zeta\right)=\sigma\right\} \label{eq:20} \\
 & = & \sup_{\sigma\in\cS}\left(1-\tfrac{\lambda}{\left\Vert \sigma\right\Vert _{q}}\right)\mathbb{E}[Z]+\tfrac{\lambda}{\left\Vert \sigma\right\Vert _{q}}\rho_{\sigma}\left(Z\right). \nonumber
\end{eqnarray}
This completes the proof of the Kusuoka representation.

In order to verify that this is the minimal Kusuoka representation consider the corresponding set $\cC\subset\cA$ in the dual, which is in view of \eqref{eq:21}
$$\cC=\left\{\zeta= (1-\lambda\,\mathbb E [\zeta^\prime])\,\ind +\lambda \zeta^\prime: \Vert\zeta^\prime\Vert_q=1, \, \zeta^\prime\ge 0\right\}$$
(cf.~\cite[Example 6.20]{SDR}). For $Z\in L_p(\O,\F,P)$ consider the function $g_Z(\zeta):=\lan \zeta, Z \ran$ as in Definition~\ref{def:expos}.
It holds that $$\sup_{\zeta\in\cA}g_Z\left(\zeta\right)=\sup_{\zeta^\prime\ge 0,\,\Vert \zeta^\prime\Vert_q\le 1}(1-\lambda\,\mathbb E[\zeta^\prime])\mathbb E[Z]+\lambda\mathbb E[\zeta^\prime Z]=\mathbb E[Z]+\lambda \sup_{\Vert \zeta^\prime\Vert_q\le 1} \mathbb E[\zeta^\prime(Z-\mathbb E[Z])].$$
For $1<p<\infty$ the latter supremum is {\em uniquely} attained at $$\bar\zeta^\prime=\frac{\left(Z-\mathbb E [Z]\right)_+^{p-1}}{\Vert\left(Z-\mathbb E [Z]\right)_+^{p-1}\Vert_q},$$ such that $g_Z$ attains its {\em unique} maximum at \begin{eqnarray}
\bar\zeta=(1-\lambda \mathbb E [\bar\zeta^\prime])\ind+\lambda\bar\zeta^\prime.
\end{eqnarray}
It follows that $\cC$ consists of exposed points only, that is $\cC=\ex(\cA)$. This proves that $\cC$ is the minimal Kusuoka representation, and thus \eqref{eq:3}.
\fb\end{proof}
\begin{cor}[cf.~\cite{mor12}]
For $p=1$ the minimal Kusuoka representation of the absolute semideviation
$\rho\left(Z\right):=\mathbb{E}[Z]+
\lambda\,\mathbb{E}\left[\left(Z-\mathbb{E}[Z]\right)_+\right]$, with $\lambda\in[0,1]$,
is
\begin{equation}
\rho\left(Z\right)=\sup_{\kappa\in\left[0,1\right]}\left\{
\left(1-\lambda \kappa\right)\AVaR_0+
\lambda \kappa\,\AVaR_{1-\kappa}\left(Z\right)\right\}.
\label{eq:3-2}
\end{equation}
\end{cor}
\begin{proof}
Note that the supremum in \eqref{eq:5} is attained at $\zeta=\frac{\left(Z-\mathbb{E}[Z]\right)_{+}^{0}}{\left\Vert \left(Z-\mathbb{E}[Z]\right)_{+}^{0}\right\Vert _{\infty}}$,
and in \eqref{eq:4} thus at $\zeta=\frac{\left(Z-\mathbb{E}[Z]\right)_{+}^{0}}{\left\Vert \left(Z-\mathbb{E}[Z]\right)_{+}^{0}\right\Vert _{1}}$,
which is a function of type $\zeta=\frac{\ind_B}{P(B)}$ (choose $B=\left\{ Z>\mathbb{E}[Z]\right\} $
to accept the correspondence). Next observe that $\sigma_\zeta(u)=\VaR_{u}(\zeta)=\frac{1}{1-\alpha}\ind_{\left[\alpha,1\right]}\left(u\right)$
(for $\alpha=P\left(B^{\complement}\right)$), which is the spectral
function of the Average Value-at-Risk of level $\alpha$. It follows
that
\begin{eqnarray*}
\rho\left(Z\right) & = & \sup_{\sigma=\frac{1}{1-\alpha}\,
\ind_{\left[\alpha,1\right]}}\left\{\left(1-\frac{\lambda}{\left\Vert \sigma\right\Vert _{\infty}}\right)\mathbb{E}[Z]+\frac{\lambda}{\left\Vert \sigma\right\Vert _{\infty}}\rho_{\sigma}\left(Z\right)\right\},\\
 & = & \sup_{\alpha}\left\{\left(1-\lambda\left(1-\alpha\right)\right)\mathbb{E}[Z]
 +\lambda\left(1-\alpha\right)\AVaR_{\alpha}\left(Y\right)\right\},\\
 & = &\sup_{\kappa\in (0,1)}\left\{\left(1-\lambda\kappa\right)\mathbb{E}[Z]
 +\lambda\kappa\,\AVaR_{1-\kappa}\left(Z\right)\right\}.
\end{eqnarray*}
That is, the set
$
\cM:=\cup_{\kappa\in (0,1)}\left\{(1-\lambda \kappa)\delta_0+\lambda\kappa \delta_{1-\kappa}\right\}
$
is a Kusuoka set and its closure is the minimal Kusuoka set.
\fb
\end{proof}

\section{Conclusion}
\label{sec:concl}
In this paper we address the Kusuoka representation of law invariant coherent  risk measures,  introduced in \cite{Kusuoka}. In general many representations may describe the same risk measure, but we demonstrate that there is a minimal representation available, and this minimal representation is moreover unique.

The minimal representation can be extracted by identifying the set of exposed points of the dual set and then applying the corresponding one-to-one transformation. Moreover, it turns out that the minimal representation only consists of  measures   which are nondominated in first order stochastic dominance. This relation, as well as the convex  order stochastic dominance relation can be employed to identify the necessary measures. This is demonstrated in   two considered examples.

We also consider  Kusuoka representations on general probability spaces, which potentially contain atoms. The gap for law invariant risk measures on spaces, which contain -- or do not contain -- atoms, is clarified and sufficiently described.

\end{document}